\documentclass[reqno]{amsart}
\usepackage[utf8]{inputenc}
\usepackage{times,amsfonts,mathabx, amsmath, amssymb,amsrefs,mathrsfs,tikz, amsthm}
\usetikzlibrary{decorations,decorations.pathmorphing}
\usepackage[left=2.5cm, right=2.5cm, top=2.5cm, bottom=3cm]{geometry}
\usepackage{dsfont}
\makeatletter
\newcommand\HUGE{\@setfontsize\Huge{45}{54}}
\makeatother
\usepackage{xcolor}
\usepackage{comment}

\newcount\dotcnt\newdimen\deltay
\def\Ddot#1#2(#3,#4,#5,#6){\deltay=#6\setbox1=\hbox to0pt{\smash{\dotcnt=1
\kern#3\loop\raise\dotcnt\deltay\hbox to0pt{\hss#2}\kern#5\ifnum\dotcnt<#1
\advance\dotcnt 1\repeat}\hss}\setbox2=\vtop{\box1}\ht2=#4\box2}

\def\clap#1{\hbox to 0pt{\hss#1\hss}}

\newcommand\Z{\mathbb{Z}}

\DeclareMathOperator{\supp}{supp}

\newcommand\R{\mathbb{R}}
\newcommand\Q{\mathbb{Q}}

\renewcommand{\i}{\mathbf{i}}

  \newcommand{\thistheoremname}{}

\newtheorem*{genericthm*}{\thistheoremname}
\newenvironment{namedthm*}[1]
  {\renewcommand{\thistheoremname}{#1}%
   \begin{genericthm*}}
  {\end{genericthm*}}

\newtheorem{lemma}{Lemma}[section]

\newtheorem{proposition}[lemma]{Proposition}

\newtheorem{corollary}[lemma]{Corollary}
\newtheorem{example}[lemma]{Example}

\theoremstyle{definition}
\newtheorem{remark}[lemma]{Remark}
\newtheorem{definition}[lemma]{Definition}
\newtheorem{question}[lemma]{Question}

\font\manfnt=manfnt

\newcommand\twoheaddownarrow{\hbox to 0pt{\raisebox{0.3ex}{$\downarrow$}}
  \hbox to 0pt{\raisebox{-0.2ex}{$\downarrow$}}\phantom\downarrow}

\title{Poisson boundary for uppper-triangular groups}

\author{Anna Erschler}
\address{A.E.: C.N.R.S., \'{E}cole Normale Superieur, PSL Research University, 45 rue d'Ulm, 75005, Paris,  France}
\email{anna.erschler@ens.fr}
\author{Josh Frisch}
\address{J.F.: C.N.R.S., \'{E}cole Normale Superieur, PSL Research University, 45 rue d'Ulm, 75005, Paris,  France}
\email{joshfrisch@gmail.com}
\author{Mark Rychnovsky}
\address{M.R. :Department of Mathematics, USC, 3620 S. Vermont Ave. CA 90089}
\email{rychnovs@usc.edu}

\date{February 2023}

\begin{document}

\maketitle

\begin{abstract} We prove that finite entropy random walks on the torsion-free Baumslag group in dimension $d=2$ have non-trivial Poisson boundary. This is in contrast with the torsion case where the situation for simple random walks on  Baumslag groups is the same as for the lamplighter groups of the same dimension.
Our proof uses the realization of the Baumslag group as a linear group.
We define and study a class of linear groups associated with multivariable polynomials which we denote $G_k(p)$.
We show that the groups $G_3(p)$ have non-trivial Poisson boundary for all irreducible finite entropy measures, under a condition on the polynomial $p$ which we call the spaced polynomial property.
We show that the Baumslag group has $G_3(1+x-y)$ as a subgroup, and that 
the polynomial $p = 1+x-y$, satisfies this property. 
Given any upper-triangle  group of characteristic zero, we prove that one of the following must hold: 1) all finite second moment symmetric random walks on $G$ have trivial boundary 2) the group admits a block, which has a $3$ dimensional wreath product as a subgroup, and all non-degenerate random walks on $G$ have non-trivial boundary. 3) $G$ has a group $G_3(p)$ as a subgroup.
We give a conjectural characterisation of all polynomials satisfying the spaced polynomial property. If this is confirmed, our result provides a characterisation of linear groups $G$ which admit a finitely supported symmetric random walk with non-trivial boundary.

\end{abstract}

\section{Introduction}

Linear groups are at the heart of the study of the Poisson boundary.
The famous result of Furstenberg \cite{furstenberg63} describes the Poisson boundary of semi-simple Lie groups (for measures that are absolutely continuous with respect to the Haar measure). 
He also proved an analogous result in the discrete setting for certain special measures \cite{furstenberg70}.
 More generally, 
Theorem 10.3 in Kaimanovich \cite{kaimanovich2000} provides a complete description of the Poisson boundary for finite first moment random walks on discrete subgroups of a semi-simple Lie group. Additionally, the theorem in Section 10.7 of that paper identifies the Poisson boundary with the exit measure on the Furstenberg boundary under an assumption of measure decay.

The question of the boundary for solvable linear groups is discussed in \cite{furstenberg73}. In particular, see the “open question” about groups of $2\times 2$ upper triangular matrix groups on page 215 of Furstenberg \cite{furstenberg73}. For symmetric measures on a connected amenable Lie group that are absolutely continuous with respect to the Haar measure, Birg{\'e} and Raugi’s result \cite{BirgeRaugi} shows that the boundary is trivial. However, for discrete solvable groups, the situation is quite different. Amenable groups can have a non-trivial boundary even for simple random walks on them. This happens in particular for solvable linear groups, with the first examples of this kind being Lamplighter groups \cite{kaimanovichvershik}.

The last few decades have seen significant progress in understanding the Poisson boundaries of groups. See \cite{furman} for an overview of topics including random walks on semi-simple Lie groups and discrete linear groups. For a survey on the Poisson boundary, see also \cite{erschler2010}, \cite{zhengicm}.
These papers discuss in particular various classes of groups, which are far from being  solvable. 
We also mention various applications of boundary triviality to other asymptotic properties of groups. For example, boundary triviality, based on   the entropy criterion,  is used in \cite{bartholdivirag}, \cite{kaimanovich}, \cites{BartholdiKaNe, AmirAngelVirag} to prove amenability for certain families of non-elementary amenable groups. See also \cites{erschlersubexp, erschlerzheng} which use boundary nontriviality to provide lower bounds for the growth of groups.

However, the basic question of boundary triviality/non-triviality has remained open for the very classical family of linear groups. In \cite{erschlerfrisch}, the first two named authors classified which finitely generated solvable linear groups over characteristic $p>0$ fields have non-trivial Poisson boundary for simple random walks. 
In this paper, we address the question of finitely generated linear groups over fields of characteristic $0$. 
We emphasize that there is not even a conjectural characterisation for general solvable groups.

We begin by formulating our result in the particular example of torsion-free Baumslag groups. Recall that wreath products $\mathbb{Z}^d \wr \mathbb{Z}$ and Lamplighter groups $\mathbb{Z}^d\wr \mathbb{Z}/2\mathbb{Z}$ were the first examples (due to Kaimanovich and Vershik \cite{kaimanovichvershik}) to show that amenable groups can admit symmetric random walks (in particular simple random walks) with non-trivial Poisson boundary. This happens for such wreath products for $d\ge 3$. On the other hand, for $d=1$ or $2$, such wreath products provide examples of groups of exponential growth with trivial boundary for simple random walks. These groups are finitely generated but not finitely presented. The Baumslag group $B_1(\mathbb{Z})$ (see Definition \ref{def:baumslaggroups}) was constructed in \cite{BaumslagB1} as an example of a finitely presented group that admits an Abelian group of infinite rank as a normal subgroup. Its $d$-dimensional analog $B_d(\Z)$ (also defined in Definition \ref{def:baumslaggroups}) is another example illustrating Baumslag’s theorem \cite{Baumslag} stating that any finitely generated metabelian group can be embedded into a finitely presented metabelian group. We mention that this theorem was apparently independently proven by Remeslennikov \cite{Remeslennikov75}, see Theorem 13.1 of that paper and the references therein for earlier announcements of this result. In the case of wreath products, Baumslag’s construction leads to the groups $B_d(\mathbb{Z})$ and $B_d(\mathbb{Z}/2 \mathbb{Z})$.

We call $B_d(\mathbb{Z})$ a torsion-free Baumslag group. For $d=1$, it is clear that $B_d(\mathbb{Z})$ and $B_d(\mathbb{Z}/2 \mathbb{Z})$ have trivial boundary for simple (and for any finite second moment) random walks since the groups are quotients of $\mathbb{Z}^2\wr \mathbb{Z}$. For $d\ge 3$, it is proven in \cite{erschlerliouv} that any non-degenerate finite entropy random walk on $B_d(\mathbb Z/p\mathbb Z)$ or $B_d(\mathbb{Z})$ has non-trivial boundary. The non-degenerateness assumption (that the support generates the group as a semigroup) can be replaced by the irreducibility assumption. A recent result of the first two authors (Example 2.9 in \cite{erschlerfrisch}) shows that any finite second moment random walk on $B_d(\mathbb{Z}/p\mathbb{Z})$ has trivial Poisson boundary.

Wreath products not only embed into Baumslag groups but also exhibit similar behavior to these groups for some asymptotic properties. Centered finite second-moment random walks on the wreath product $\Z^d \wr \Z / p \Z$ and the related Baumslag group $B_d(\Z / p \Z)$ both have nontrivial Poisson boundary if and only if $d \geq 3$. A similar statement is true for $\Z^d \wr \Z$ and the related Baumslag group $B_d(\Z)$ in the previously known cases ($d \neq 2$). We consider the only remaining case $B_2(\Z)$. As previously stated, centered, finite second-moment random walks on the related groups $\Z^2 \wr \Z$ and $B_2(\Z / p \Z)$ have trivial Poisson boundary; In contrast, we prove:

\begin{namedthm*}{Theorem A} Any finite entropy irreducible random walk on $B_2(\mathbb{Z})$ has non-trivial Poisson boundary. \end{namedthm*}

The analogous claim also holds for a subgroup of $B_2(\mathbb{Z})$, which we call the restricted (torsion-free) Baumslag group $\bar{B}_2(\mathbb{Z})$. This restricted Baumslag group is a particular case of groups in the family $G_3(p)$. Given a polynomial in $X_1$, $X_2$, $X_3$, we denote by $G_3(p)$ the upper-triangular group over the ring $\mathbb{Z}(X_1^{\pm 1},X_2^{\pm 1},X_3^{\pm 1})/(p(X_1,X_2,X_3))$ generated by the following four matrices: \begin{equation} 
\delta=\left( \begin{array}{ccc} 
1 & 1 \\
0 & 1 
\end{array} \right),
\medspace \medspace 
M_{x_i}= \left( \begin{array}{ccc} 
1 & 0 \\
0 & X_i
\end{array} \right), \end{equation} Where $i=1$, $2$, or $3$.

(We refer to  Definition \ref{def:gkI} and Section \ref{sec:prelim} for a more general class of groups and their basic properties).
We show that the groups $G_3(p)$ are crucial for our understanding
of the Liouville property for linear groups of characteristic zero, see Theorem C.

\begin{definition} \label{def:SPP}
A Laurent polynomial $p$ in
$x_1$, \dots, $x_d$ satisfies the spaced polynomial property if there exists some $N>0$ so that for any non-zero Laurent polynomial $u$ with coefficients in $0, \pm 1$, $p$ does not divide $u(x_1^N,\dots,x_d^N)$. 
\end{definition}

We prove
\begin{namedthm*}{Theorem B}
\label{prop:generalargument}
Assume that  a polynomial $p(x_1,x_2,x_3)$ has the spaced polynomial property. Then any finite entropy irreducible measure $\mu$ on $G_3(p)$ has nontrivial Poisson boundary.  
\end{namedthm*}

We explain in Remark \ref{rem:charpG3q} that the characteristic $p$ analogue $G_3(q, \mathbb{Z}/p \mathbb{Z})$ of groups $G_3(p)$ has trivial boundary for any non-zero polynomial $q$.

Before formulating our result for linear groups in characteristics $0$, we recall some results from \cite{erschlerfrisch}, where we characterise Liouville simple random walks for group of positive characteristic.
First we mention that in Thm B
of \cite{erschlerfrisch}, we proved a reduction of the Liouville property for a random walk on a group $G$ of upper-triangular matrices (over an arbitrary field $k$) to the Louiville property for random walks on
"valid blocks", a family of naturally defined $2\times 2$ matrices, associated to the linear group. See Subsection
\ref{subsec:bb} and Definition \ref{def:blocks} in the current paper for a review.
Using this reduction, we have characterised the Liouville property for simple
random walks on any linear group of characteristics $p$. A finite second moment random walk has the Liouville property 
if all valid blocks have dimension $\le 2$. Otherwise, if there exists at least one valid block of dimension $d\ge 3$, any finite entropy non-degenerate random walk has non-trivial boundary.
While even in positive characteristic the groups in question might have richer asymptotic geometry than wreath products, the
situation with the Liouville property turns out to be in a strong analogy with the example of Lamplighter groups (since $\mathbb{Z}^d \wr \mathbb{Z}/p\mathbb{Z}$ have dimension $d$, and a well-known argument of \cite{kaimanovichvershik} shows that the boundary is non-trivial if and only if $d\ge 3$).

\begin{namedthm*}{Theorem C} Let $G$ be a finitely generated linear group
of characteristic $0$. Then at least one of the following properties hold.
\begin{itemize}
    \item $G$ contains a free non-Abelian group as a subgroup. 
    \item $G$ has a finite index subgroup, of uppertriangular matrices, such that at least one valid block
    contains $\mathbb{Z}^3 \wr \mathbb{Z}$ as a subgroup.
    \item $G$ has a finite index subgroup, of uppertriangular matrices, such that at least one block
    contains $G_3(p)$ as a subgroup, for a polynomial $p$ with integer coefficients
    in $3$ variables which is irreducible over $\mathbb{Z}$ and is not generalized cyclotomic (see Definition \ref{def:gencyc}).   
    \item All finite second moment measures on $G$ have trivial Poisson boundary.
\end{itemize}
\end{namedthm*}

By Tits alternative, the first case is equivalent to the non-amenability of $G$, and it is well-known that for non-amenable groups all irreducible measures have non-trivial boundary (see \cites{azencott, furstenberg73}; see more on this in the section \ref{sec:prelim}). If the group is amenable, it is known that there exists an irreducible measure (which can be chosen to be symmetric and have full support $\supp \mu =G$) such that the Poisson boundary is trivial \cites{kaimanovichvershik, rosenblatt74}.

In the second case the boundary of any non-degenerate finite entropy random walk on $G$ is non-trivial, as shown in \cite{erschlerfrisch}.

Case 3 does not have an analog in characteristic $p$ and is a primary motivation for this paper. 

We mention that it is easy to see that if a polynomial (in several variables) is a generalized cyclotomic, then it does not satisfy the spaced polynomial property (see Lemma \ref{lem:easycyclotomic}). We conjecture that the converse is true

\begin{question}
Let $p$ be a non-zero irreducible polynomial in several variables, which is not a generalized cyclotomic. Is it true that $p$ must satisfy the spaced polynomial property?
\end{question}

Observe that if the answer to this question is positive, then by combining Theorem B and Theorem C we get a characterisation of finitely generated linear groups that admit a finitely supported measure with non-trivial Poisson boundary.

{\bf Acknowledgements.} 
This project has received funding from the European Research Council (ERC) under
the European Union's Horizon 2020 
research and innovation program 
(grant agreement
No.725773). The work of the first named author was also supported by ANR-22-CE40-0004   GOFR
The work of the second named author was also supported by NSF Grant DMS-2102838.
The work of the third named author was also supported by
the Fernholz Foundation and the NSF Grant No. DMS-1664650.
The first named author is grateful to Vadim Kaimanovich for many discussions on the history of boundary theory.

\section{Preliminaries}\label{sec:prelim}

\subsection{Notation and some classes of groups we consider}

\begin{definition}\label{def:gkI}
Let $X_1, X_2, \dots, X_k$ be a finite set of formal variables. 
Let $I$ be a prime ideal in the ring $R$
of Laurent polynomials over $R$ in $X_1,\dots,X_k$. Then $G_k(I)$ is the group generated by the following matrices

\begin{equation}
   \delta=\left( \begin{array}{ccc}
1 & 1  \\
0 & 1
 \end{array} \right),
 \medspace \medspace
 M_{x_i}=
\left( \begin{array}{ccc}
1 & 0  \\
0 & X_i\\
 \end{array} \right).
\end{equation}
over the ring $\Z(X_1^{\pm 1},\dots,X_k^{\pm })/I$.
Here $1\le i \le k$.
\end{definition}
In the sequel for our results, we assume that $R=\mathbb{Z}$. We will also mention some examples for $R=\mathbb{Z}/p\mathbb{Z}$.

If  the ideal $I$ is generated by a polynomial $p$, we also use the notation $G_k(I) =G_k(p)$.
Also, given several polynomials with integer coefficients in $k$ variables so that 
$$
I= (p_1, \dots , p_m),
$$
we can use the notation $G_k(I) =G_k(p_1, \dots, p_m)$.

\begin{remark}\label{rem:normalformgkI}
[Normal form of the elements in $G_k(I)$]
Observe that elements of $G_k(I)$ are of the form

$$
\left( \begin{array}{ccc}
1 & f(X_1,X_2,\dots, X_k)  \\
0 & X_1^{i_1} X_2^{i_2} \dots X_k^{i_k}\\
 \end{array} \right)
$$

Here $i_1$, \dots, $i_k$ are integers.
 $f$ is a polynomial in $X_i^{\pm 1}$ considered modulo $I$.
 Any element of the form above belongs to $G_k(I)$.
We note that the additive group with values in the upper right corner is isomorphic to the Abelian group of unipotent elements in $G_k(I)$ under multiplication.
Also note that there is a homomorphism from our group to $\mathbb{Z}^k$, given by the mapping to the monomial in  the lower right corner.
\end{remark}

\begin{remark} If $I \cap \Z =0$, we can tensor our ring by $\mathbb{Q}$ and consider our matrices over the field
$\mathbb{Q}(X_1^{\pm 1},\dots,X_k^{\pm })/(I \otimes \mathbb{Q})$.
This will be our convention for the particular case of such groups: for example, for the restricted Baumslag groups discussed below.
\end{remark}

We recall that a wreath product of $A$ and $B$ is a semi-direct product 
$G=A \wr B=
A \ltimes \sum_A B$,
where $A$ acts by shifts on the index set. We use the notation $A \wr B$ (as e.g. in \cite{kaimanovichvershik}, while some papers use
the notation $B\wr A$ for the acting group $A$).
In the case of $B=\mathbb{Z}/2\mathbb{Z}$ such groups are called lamplighter groups, since the element of  
$\sum_A B$ can be described as "lamps" at the points of $A$.

The elements of the lamplighter group $\Z^k \wr \Z$ can be represented as
\[\left( \begin{array}{ccc}
1 & f(x_1,...,x_k) \\
0 & x_1^{i_1} \dots x_k^{i_k}\\
\end{array} \right) \]

where the element in the upper right corner $f(x)$ is a Laurent polynomial in $x_1$, \dots, $x_k$.
Here the monomial in the lower right entry of the matrix corresponds to the projection to the base group $\mathbb{Z}^k$.
In the "lamp" interpretation, mentioned above, non-zero monomials of $f$ correspond to positions in $\mathbb{Z}^d$ where the lamp is lit.

In particular, in our notation if $I=0$ is the trivial ideal of $\Z(X_1^{\pm 1},\dots,X_k^{\pm })$, then
$G_k(0) =G_k(I)$ is isomorphic to the $k$ dimensional lamplighter group $\Z^k \wr \Z$. 

\begin{definition}[Baumslag groups] \label{def:baumslaggroups}

The Baumslag groups $B_d(\mathbb{Z}/p\mathbb{Z}) \subset 
GL_2((\Z/p\Z)(Y_1,..., Y_d)$ and
$B_d(\mathbb{Z}) \subset GL_2((\Z)(Y_1,...,Y_d)$, are the groups
generated by $2\times 2$ matrices
of the form 
$$
\delta=\left( \begin{array}{ccc}

1 & 1  \\

0 & 1
 \end{array} \right),
 \medspace \medspace
 M_{y_i}=
\left( \begin{array}{ccc}
1 & 0  \\
0 & Y_i\\
 \end{array} \right), 
\medspace \medspace
M_{y_i+1}=
 \left( \begin{array}{ccc}
1 & 0  \\
0 & Y_i+1\\
 \end{array} \right).
$$
over the field of rational functions $(\Z/p\Z)(Y_1,...Y_d)$ and $(\Q(Y_1,...Y_d)$ respectively
($1 \le i \le d$). 
\end{definition}

We consider  the homomorphisms from $B_d$  to $\mathbb{Z}^d$ (and to $\mathbb{Z}^{2d}$), defined  for $\mathbb{Z}$ and $\mathbb{Z}/p\mathbb{Z}$ cases, that sends an element with $\prod Y_i^{\alpha_i} \prod (Y_i + 1)^{\beta_i}$ in the lower right corner to $(\alpha_1, \dots, \alpha_d)$ (correspondingly to $(\alpha_1, \dots, \alpha_d, \beta_1, \dots, \beta_d)$) and denote them as $\pi$  and $\phi$.

\begin{remark}
The group $B_d(\mathbb{Z})$ is finitely presented.
The group has the presentation
$$
\delta^{M_{y_i+1}} = \delta \delta^{M_{y_i}}
$$
$$ [M_{y_i}, M_{y_j}]=[M_{y_i+1},M_{y_j+1}]=
[M_{y_i},M_{y_j+1}]=1 = [\delta^{u}, \delta^{v}],
$$
where in the relations above $1 \le i,j \le d$;
and $u$ and $v$ are product of $M_{y_i}$ and $M_{y_i+1}$,
$1 \le i,j \le d$,
to the power $0$ or $1$.
See e.g. second claim of Lemma 5.1 in \cite{erschlerliouv}.
Here we use the notation $a^b$ for the conjugation of $a$ by $b$.
\end{remark}

\begin{remark} As explained in the introduction, Theorem A establishes that the entropy of $B_2(\mathbb{Z})$ grows linearly, which is in contrast to the sublinear growth of the entropy function for $B_2(\mathbb{Z}/p\mathbb{Z})$. 
Let us mention a property of a very different nature (not related to random walks), where $B_d(\mathbb{Z})$ and $B_d(\mathbb{Z}/p\mathbb{Z})$ have different asymptotic properties.
A result of Kassabov and Riley \cite{KassabovRiley} shows that
the Dehn function of $B_1(\mathbb{Z})$ is exponential,
while by a result of de Cornulier and Tessera \cite{deCornulierTessera}
the Dehn function $B_1(\mathbb{Z}/p \mathbb{Z})$ is quadratic.
\end{remark}

In this paper, we are interested in the case $d=2$. $B_2(\Z/p\Z)$ is the Baumslag group and $B_2:=B_2(\Z)$ is the torsion-free Baumslag group. We also consider the {\it restricted} Baumslag group which is generated by
$4$ generators
$\delta$, $M_{x_1}$, $M_{x_2}$ and $M_{x_1+1}$ which we denote by $\tilde{B}_2(\Z)$

Observe that 
\begin{equation} \label{eq:Baumslag} \tilde{B}_2(\Z)=G_3(1+x_1-x_2), \qquad B_2(\Z)=G_4(1+x_1-x_2,1+x_3-x_4) \end{equation}

Indeed, There is an isomorphism sending $M_{y_1}$ to $M_{x_1}$, $M_{y_1+1}$ to $M_{x_2}$, $M_{y_2}$ to $M_{x_3}$ and $M_{y_2+1}$ to $M_{x_4}$.

\subsection{Trajectories of random walks, boundary and asymptotic entropy}

The random walk $(G,\mu)$, defined by a probability measure $\mu$ on a countable group $G$, is a Markov chain with state space $G$, with transition probabilities from $g$ to $gh$ being equal to $\mu(h)$, for all $g,h \in G$.
As we have mentioned in the introduction, a random walk is called irreducible if its support generates $G$ as a group. Any random walk is irreducible if we consider it as a random walk on a group generated by its support.

We recall the definition of the entropy of random walks, also called asymptotic entropy.

\begin{definition}
Consider a countable group $G$ and a probability measure $\mu$ on $G$.
The {\it entropy} of a random walk $(G,\mu)$,
is defined as  the limit of the normalised entropies of its convolution: 
$$
h=\lim_{n\to \infty} H(n)/n,
$$
where
$H(n)=H(\mu^{*n}) =\sum_{g\in G} -\log (\mu^{*n}(g)) \mu^{*n}(g)$. 
\end{definition}

The notion of the entropy of random walks is due to Avez \cite{avez}, who proved that for finitely supported random walks the boundary is trivial if $h=0$. This finite support assumption turned out to be inessential, and the sufficient condition above is also necessary. This is due to Kaimannovich-Vershik \cite{kaimanovichvershik}[Thm 1.1] and Derriennic \cite{derriennic}:

\begin{namedthm*}{Entropy criterion}
 \label{th:entropytoboundary}
Let $G$ be a discrete countable group, and let $\mu$ be a finite entropy measure on $G$. The Poisson boundary of $(G,\mu)$ is trivial if and only if the entropy of the random walk $h$ is zero.
\end{namedthm*}

In other words, the Poisson boundary of $(G,\mu)$ is non-trivial 
if and only if the entropy of the $n$ step distribution of the  random walk $H(n)$ is linear in $n$.

\begin{definition}
We say that a function $F:G \to \R$ is $\mu$-{\it harmonic}, if for all $g\in G$ it holds
$F(g) = \sum_{h\in G} F(gh) \mu(h)$.
\end{definition}

The Poisson boundary 
can be defined in terms
of bounded harmonic functions
on the subgroup of $G$,
generated by the support of $\mu$ and with values in $\mathbb{R}$.

Non-triviality of the Poisson boundary is equivalent to the  existence of  non-constant bounded harmonic functions on the group, generated by the support of $\mu$.  

\begin{definition}
Given a probability measure on $G$, we say that $(G,\mu)$ satisfies Liouville property if any bounded harmonic function on the group, generated by the support of $\mu$, is constant.
\end{definition}

There are several ways to define the Poisson boundary (see \cite{kaimanovichvershik}, Section 0.3). For general Markov chains, some of these definitions lead to different notions, but in the case of random walks on groups, these notions are equivalent.

Given an equivalence relation on a probability space $X$, its measurable hull is a $\sigma$-algebra  of all measurable subsets of the path space which are unions  of the equivalence classes. 

\begin{definition}[Poisson boundary]

Consider the space of one-sided infinite trajectories $G^\infty$.
We say that two infinite one-sided trajectories
$X$ and $Y$ are equivalent if they coincide, up to a possible time shift,  after some instant. This means that 
there exist $N,k\ge 0$ such that $X_i=Y_{i+k}$ for all $i>N$. Consider the measurable
hull of this equivalence relation in $G^\infty$. 
The quotient of 
the probability space $G^\infty$
by the obtained equivalence relation is called the {\it Poisson boundary} (also called the {\it Poisson-Furstenberg boundary}).
\end{definition}

In the definition above, if we consider the equivalence relation $X_i=Y_{i}$ (without allowing the time shift) we obtain the definition of the tail boundary. While for general Markov chains the tail boundary can be larger, for random walks on groups these notions are equivalent.

In this paper, we do not use particular definitions of the boundary. The proof of our main results will rely on new lower bounds for the asymptotic entropy of the random walks.

Now we mention that there are also many equivalent definitions of amenability of groups:
in terms of isoperimetric inequalities (F{\o}lner criterion), existence of invariant means, and many others. Below we recall the one in terms of return probability of random walks (Kesten's criterion, see e.g. Corollary 12.5 in \cite{woessbook}):
\begin{definition}
Let $\mu$ be a symmetric probability measure on $G$, whose support generates $G$.
A group $G$ is said to be amenable if the $n$ step return probabilities of the random walk corresponding to $\mu$ satisfy
$$
\mu^{*2n}(e)
$$
is subexponential in $n$.
\end{definition}
The property in the definition above does not depend on the choice of $\mu$.
If the group is amenable, all irreducible symmetric random walks on $G$ have subexponential decay of return probabilities. The assumption of symmetricity is essential, it is easy to see that finitely supported non-centered random walk on $\mathbb{Z}$
(and many other groups)
have exponential decay of probability to return to the origin.
If the group is non-amenable, then for any irreducible random walk, symmetric or not,
$\mu^{*n}(e)$ has exponential decay.

We have mentioned already that any irreducible random walk on a non-amenable group has non-trivial boundary. Indeed, the result of Azencott (Proposition II.1 \cite{azencott})
shows  moreover that almost surely the stabilizer of a point on the Poisson boundary for a random walk on $G$ has the fixed point property with respect to $G$. Thus, if the Poisson boundary for a random walk defined by an irreducible measure is trivial, then $G$ has a fixed point property with respect to $G$. The latter property  is equivalent to amenability.
See also \cite{furstenberg73}, where Furstenberg proves non-triviality of an irreducible random walk on a non-amenable group in his proof of the transience of the random  walk (see page 213  where he states: "we notice in connection with the first corollary that what has been shown is
that if G is nonamenable and the support of $\mu$ generates G then there exists a nontrivial
$\mu$-boundary"). Non-triviality of the Poisson boundary indeed implies transience, but the latter property is much weaker. We mention that a finitely generated infinite group is transient unless it has  a finite index subgroup isomorphic to $\mathbb{Z}$ 
or $\mathbb{Z}^2$. However, there are a vast variety of groups admitting (irreducible) measures with trivial boundary: all groups of subexponential growth, but also some groups of exponential growth; for example, the  already mentioned  two dimensional lamplighter, as well as the one-dimensional Baumslag groups or the two dimensional torsion-by-Abelian Baumslag group $B_d(\mathbb{Z}/p\mathbb{Z})$.

\section{Cube independence and estimates of  $\Delta$-restriction entropy}

While all amenable groups are known to admit non-degenerate measures with a trivial Poisson boundary \cites{kaimanovichvershik, rosenblatt74}, for many amenable groups such measures can not be chosen to have finite entropy.
This is the case for $d\ge 3$ dimensional wreath products
\cite{erschlerliouv} and more generally for all linear groups of characteristic $p$ admitting simple random walks with non-trivial boundary \cite{erschlerfrisch}.
Non-triviality of the boundary can be seen in this case by showing that asymptotic entropy grows linearly, and the latter by providing a lower bound for $\Delta$-restriction entropy.
First, we recall this notion (Definition 3.10 in \cite{erschlerfrisch}):

\begin{definition}[$\Delta$-restriction  entropy]\label{def:deltarestrictionentropy}
Given a group $G$, a probability measure $\mu$ on $G$ and
a finite set $\Delta 
\subset \supp (\mu)$,
we define the $\Delta$-restriction  entropy
$H_{\Delta}(n)$
as follows.
We consider an $n$-step trajectory $X_n$ of $(G,\mu)$.
Then we take the conditional entropy of $X_n$, after
conditioning on all increments except those that are in $\Delta$.
\end{definition}

Now we explain how the notion of cube independence can be used to obtain
lower bounds for $\Delta$-restriction entropy.

\begin{definition}\label{def:cubeind}
We say that a sequence of elements $\gamma_1$, $\gamma_2$ \dots $\gamma_n \in G$ is cube independent if the elements below are pairwise distinct
$$
\gamma_1^{\varepsilon_1} \gamma_2^{\varepsilon_2} \dots \gamma_n^{\varepsilon_n } \,
$$
where $\varepsilon_i =1$ or $0$. 
\end{definition}

In the following definition, we introduce a convention of how to speak about the same images of a mapping, which is not necessarily a group 
homomorphism.

\begin{definition}[Same image under $\pi$]
 Given 
 a subset $S\subset G$ and
 a subgroup $H$ such that  $ S\subset H\subset G$, a mapping $\pi:S \to X$ and elements 
 $\delta_1,  \delta_2 \in S$. We say that $\delta_1$
 and $\delta_2$ have ``the same image under $\pi$":
 if $\pi$ factors through a group homomorphism  which sends $\delta_1$ and $\delta_2$ to the same element, in other words, there exists a group $H'$, a group homomorphism $\pi': H \to H'$ and a mapping $\pi" :H' \to X$ such that
 $\pi = \pi" (\pi')$ and $\pi'(\delta_1) =\pi'(\delta_2)$.
 \end{definition}

 Given a group $G$ and a subset $\Delta$ of cardinality two, we say that a subset $S$ is adapted to  
 $\Delta$ if $S$ is the preimage of a subset $S'$ of the quotient $G/ \langle \delta_1\delta_2^{-1} \rangle$, where $\langle \delta_1\delta_2^{-1} \rangle$ denotes the normal subgroup generated by $\delta_1\delta_2^{-1}$.

 If $X$ in the definition above is a group and $\pi$ is a group homomorphism, then the definition above is the same as saying $\pi(\delta_1)=\pi(\delta_2)$.

\begin{definition}[Cube independent property along the image $\pi(S)$]
Take a two element set $\Delta =\{\delta_1, \delta_2)$,  assume that these two elements
have the same image under $\pi$ and that the set $S$ is adapted to $\Delta$.
 We say that $\Delta$ has the cube independent property along 
 the image of $\pi(S)$ if 
 for any $h_1$, $h_2$, 
 \dots $h_k$ 
 such that $\pi(h_1 \delta_{1} \dots h_j \delta_{1})$ 
 are distinct for $1\le j \le k$ the following $2^k$
 elements
 $$
 h_1 \delta_{i_1} h_2 \delta_{i_2} \dots h_k \delta_{i_k}
 $$
 where $i_s = 1$ or $2$ for all $1 \le s \le k$, are distinct elements of $S$.
\end{definition}

\begin{remark} \label{rem:cubeindependent}
Let $\pi$ be a map, and let $\delta_1$ and $\delta_2$ have the same image under $\pi$. Set $\bar{\delta} =\delta_1^{-1} \delta_2$. 

Set $r_1=h_1 \delta_1$, $r_2 = h_1 \delta_1 h_2 \delta_1$, \dots , 
$r_k= h_1 \delta_1 h_2 \delta_1 \dots h_k \delta_1$. 
The cube independent property along the image $\pi(S)$ 
says that for any $h_1, h_2, ... , h_k$ above so that $\pi(r_1), \pi(r_2),... \pi(r_k)$ are distinct, we have that 
$$h_1 \delta_{i_1} h_2 \delta_{i_2} ... h_k \delta_{i_k}$$ are pairwise distinct for $(i_1,...,i_k) \in \{1, 2\}^k$.

Setting $\gamma_i = r_i \bar{\delta} r_i^{-1}$, the usual cube independence property
from Definition \ref{def:cubeind}
says that 

$$\gamma_1^{\varepsilon_1} ... \gamma_k^{\varepsilon_k}$$
are pairwise distinct for $(\epsilon_1,...,\epsilon_k) \in \{0, 1 \}^k$.

These two properties are equivalent for any given $h_1,...,h_k$ because

$$ \gamma_1^{\varepsilon_1} ... \gamma_k^{\varepsilon_k}r_k = h_1 \delta_{i_1}...h_k \delta_{i_k}$$

for $i_j = \epsilon_j + 1$.

Thus the cube independent property along the image $\pi(S)$ is equivalent to the statement that for any $r_1,... r_k \in S$ with distinct images under $\pi$, the conjugates $\gamma_i = r_i \bar{\delta} r_i^{-1}$ of $\bar{\delta}$ satisfy the usual cube independence property.

\end{remark}

We recall that a subset $S$ of a finitely generated group is (left) syndetic if some finite neighborhood of $S$ in the left invariant
Cayley graph of $G$ is equal to $G$. 
(See e.g. \cite{cellularautomatabook}; the notion of syndeticity in his definition (3.39) corresponds to right syndecicity; our convention is to consider left invariant Cayley graphs, so that we will need the notion of left syndeticity).  It is clear that a subgroup is left syndetic if and only if it has a finite index in $G.$

\begin{example}[Wreath products (Lamplighters), cube independence along $\pi$]
$G:A\wr B$, $H=G$. Consider the group homomorphism $\pi: G \to A$. Let $
\Delta=\{\delta_1, \delta_2\}$. 
\begin{itemize}

\item Set $\delta_1=e$ and $\delta_2 =\delta$ ("the lamp").
Then $\Delta$ is cube independent along the image $\pi(G)$.

\item If $A=\mathbb{Z}^d$ (or any other orderable group) and we have two elements $\delta_1$  and $\delta_2$
with the same projection
$\pi(\delta_1)= \pi(\delta_2)$ to $A$. Then $\Delta$ is cube independent along $\pi$.

\item In general, $\Delta$ is not necessarily cube independent over the image of $\pi$. An example $A$ is a finite group
(or any other group with torsion), and $B=\mathbb{Z}/2 \mathbb{Z}$. $\delta_1 =e$ and $\delta_2$ has trivial projection to $A$ and consists of a configuration of 2 lamps. 

\item For any $A$  and $\delta_1, 
\delta_2 \in G$ such that $\pi(\delta_1)=\pi(\delta_2)]$
there exists a left syndetic subset $S\subset G$
such that  $\Delta=\{\delta_1, \delta_2\}$ is cube independent along 
$\pi(S))$.

\end{itemize}
\end{example}

\begin{proof}

To prove the first claim, observe that $r_i \delta r^{-1}_i$ corresponds in the normal form to the lamp at $\pi(r_i)$. 
Taking elements with different projections to $A$, we see that $r_i \delta r_i^{-1}$ generates an Abelian group, which is the product of the cyclic groups generated by each $r_i \delta r_i^{-1}$.

To prove the second claim, fix some order on the group $A$. For example, if $A=\mathbb{Z}^d$ we can take a lexicographic order. Let $\delta' =\delta_1^{-1} \delta_2 \ne e$.
Let us show that if $r_i$ have distinct projection to $A$, then $r_i \delta' r_i^{-1}$ are cube independent.
Take maximal $r_i$ with respect to our order and assume that in exactly one of the products
we have $r_i \delta' r_i^{-1}$. In other words, in the Abelian group, it corresponds to $\delta'$ placed at $r_i$. Consider the support of $\delta'$ and let $l$ be the maximal element of this support. Observe that in the product of terms $(r_i \delta' r_i^{-1})^{\epsilon_i}$, where $\epsilon_i \in \{0, 1\}$ the maximal position where a lamp is lit is at $r_i + l$ exactly if $\epsilon_i = 1$.
Now we can remove this term from the product and repeat the argument using the next largest $r_i$ to find every $\epsilon_i$. 

To prove the third claim observe that the cube independence would imply that the product of shifts of $\delta'$ over all elements of $A$ is not $e$. But the value at each element of the support is $1+1=0$ mod $\mathbb{Z}/\mathbb{Z}$.

Finally, fix a generating set of $A$.
To prove the fourth claim take $S$ such that the distance between two points of $S$ is $>L$, where $L$ is the maximal length of the support of $\delta'$.
\end{proof}

Take $A$ to be a group without finite index subgroups, which contains torsion elements. For example, we can take $A$ to be a Tarski monster: an infinite finitely generated group such that any proper subgroup is cyclic of order $p$. (The constructions of such groups is due to A.Yu.Olshansky \cite{olsh}). In the lamplighter group for $A$ we consider the standard two point set $\Delta=\{e, \delta_2 \}$ as in the third claim above and consider the projection $\pi$ to $A$.
By the third claim, we know that since $A$ has torsion elements, we do not have the cube independence property for the set $\Delta$ along the image $\pi(A)$.
Since $A$ does not have non proper finite index subgroups, there is no finite index subgroup $H$ where $\Delta$ might have a cube independence property along $\pi(H)$.
This shows that even in the case of lamplighters, it can be important to consider subsets (in this case syndeticity would work) and not only subgroups to ensure cube independence
along the image of the subset.
We will later discuss this context and its application for entropy estimates in the Lemma \ref{lemme:corelemma}.

\begin{lemma} \label{ex:baumslagci} [Baumslag groups, cube independence along $\pi$]
$G =B_d(R)$, where $R= \mathbb{Z}/p\mathbb{Z}$ or $R=\mathbb{Z}$.
Consider the standard homomorphism 
$$
\pi: B_d(R) \to \mathbb{Z}^d
$$
sending the matrix with a lower right entry $\prod_{i=1}^d X_i^{\alpha_i}  \prod_{j=1}^{d} (1+X_j)^{\beta_j} $ to $ \left( \alpha_1, \alpha_2, \dots , \alpha_d \right) $.
Let $\Delta=\{\delta_1, \delta_2\}$ where $\delta_1$ and  $\delta_2$ have the same projection under $\pi$.
Then $\Delta$ is cube independent along  the image $\pi(G)$.
\end{lemma}
\begin{proof}
Consider $\delta'= \delta_1^{-1}\delta_2$. Let us prove the cube independence of $\{e,\delta'\}$ along the image $\pi(G)$.
Take several elements $r_j$ with distinct projections to $\mathbb{Z}^d$
with respect to $\pi$. Consider their projection to $\mathbb{Z}^{2d}$ with respect to
$\phi$. (See the remark after Definition 2.4 for the definition of these homorphisms). Observe that the conjugates of $\delta$ by $g$ depend on $\phi(g)$ only. 

Observe that without loss of generality, we can assume that $\phi(r_i) \subset \mathbb{Z}^d \times \mathbb{Z}_+^d$. Indeed, otherwise we conjugate all $\phi(r_i)$ by a large enough positive power of $\prod_i (Y_i+1)$ so that the exponent of each $(Y_i + 1)$ is positive.  Given $f\in G$, it is clear that $\gamma_i$ are cube independent if and only if $f\gamma_if^{-1}$ are cube independent.

Now observe that $\prod Y_i^{\alpha_i}(1+Y_i)^{\beta_i}$, $\beta_i \ge 0$ has a non-zero coefficient for the  monomial
$\prod Y_i^{\alpha_i}$ and  has zero coefficients for the monomials $\prod Y_i^{\alpha'_i}$ where $(\alpha'_1, \dots , \alpha'_i)$ is smaller in the lexicographical order than  $(\alpha_1, \dots , \alpha_i)$.
We order $r_i$ in this lexicographical order.
If $\delta'=\delta$ we can look at each product in the definition of the cube independence and check whether each monomial is nonzero. If it is nonzero we know can subtract out the corresponding polynomial and continue checking monomials, to see exactly which conjugates have exponent 1. Therefore we do have the cube independence for the conjugates of $r_i$.

In the general case, having $\delta' \ne e$ with $\pi(\delta')=e$, we argue in a similar way as in the second claim of the Lamplighter example.
We consider the maximal element
in the support of $\delta'$ in the lexicographical order on $\mathbb{Z}^d$.

\end{proof}

For our proof of Lemma \ref{lemme:corelemma} it will be helpful to observe the following.

\begin{remark} \label{lem:semigroupBaum}
Let $R=\mathbb{Z}$ or $R=\mathbb{Z}/p\mathbb{Z}$, $d\ge 1$.
Let $S^+$ be a sub-semigroup of $B_d(R)$, such that the group generated by $S^+$ is not Abelian and such that 
$\pi(S^+)$ generates $\mathbb{Z}^d$ as a group. Then there exist $\delta_1 \ne \delta_2 \in S^+$ such that $\pi(\delta_1)=\pi(\delta_2)$.  
\end{remark}
\begin{proof}
Choose $s$ and $s'$ to be non-commuting elements of $S^+$.
Then $ss'\ne s's$,
$ss'$ and $s's$ are in $S^+$ and $\pi(ss')=\pi(s)+\pi(s')=\pi(s's)$. So we can put $\delta_1= ss'$ and $\delta_2=s's$.
\end{proof}

In the lemma above and the remark above we use the notation $R$ for the Abelian group (and elsewhere we usually use this notation for the corresponding ring).

For our proof of Theorem A, it is essential that the torsion-free Baumslag group also admits a homomorphism to $\mathbb{Z}^{d+1}$ (and not only $\mathbb{Z}^d$ where we will have the  cube independence of $\Delta$
along the image of $\pi$. This will be proven in Proposition $\ref{prop:ciBaumslag}$.
In contrast to the example above, the cube independence will be claimed not for $\pi(G)$, but for $\pi(H)$, where
$H$ is a finite index of subgroup $G$.

One can show that  for $R= \mathbb{Z}/p \mathbb{Z}$ there is no homomorphism  $\pi$ from $G$
(or from a finite index subgroup of $G$) to 
$\mathbb{Z}^{d+1}$  so that $\Delta$ is cube independent along $\pi$.

However, for our argument in the proof of theorem A we will use a group homomorphism $\pi$ from a finite index subgroup of $G= B_d(\mathbb{Z})$ to $
\mathbb{Z}^{d+1}$.

\begin{lemma}[Lower estimates for $\Delta$-restriction entropy] \label{lemme:corelemma}
\begin{enumerate}

\item
We assume that $S\subset G$, $\pi:S \to X$ and a two point set 
$\Delta$ whose elements have  the same projection under $\pi$ are such that $\Delta$ has a cube independent property along $\pi(S)$.
Let $\mu$ be a probability measure on $G$ such that $\Delta$ belongs to the support of $\mu$. 

For a trajectory of our random walk $X_i$, we consider the number of distinct points in the set $\{\pi(X_i)\}_{i=1}^n$,  where $X_i \in S$. If the probability that  this number, defined by a randomly chosen trajectory, is linear in $n$, is positive, 
then the $\Delta$ restriction entropy is positive.

\item In particular, the assumption (and thus the claim) of the Lemma is satisfied  if $\pi: S \to \mathbb{Z}^3$ is a surjective group homomorphism,
 $(G,\mu)$ is an irreducible random walk, $S$ is a finite index subgroup of $G$, $\Delta =\{\delta_1, \delta_2\}$,
 $\delta_1\ne \delta_2$, $\pi(\delta_1)=\pi(\delta_2)$
and  there is a cube independent property for $\Delta$  along the image $\pi(S)$.

The same holds more generally if $G$ admits  a surjective homomorphism $\pi: S \to A$, $A$ is an infinite finitely generated group which is not virtually $\mathbb{Z}$
or $\mathbb{Z}^2$, $S$ is a syndetic subset of $A$, 
$(A,\pi(\mu))$ is irreducible and
there is the cube independent property for $\Delta$  along the image $\pi(S)$.
\end{enumerate}

\end{lemma}

\begin{proof}
We know that  there exist  constants $C,p>0$ such that  with probability $p$
the number of distinct points among 
$\phi(X_i)$ is at least $Cn$.

In this case there exist $C', p'>0$ such that
the number of distinct
$\phi(X_i)$, $1 \ne i \le n$ with the increment $Y_{i+1}$ in $\Delta$, is $\ge C' n$ 
with probability $p'$. Call the claim in the previous sentence (*).
Indeed, observe that if we know the trajectory up to a time instant $k$, we know that with positive probability the next step is obtained by multiplication by an element in $\Delta$.

Let $\chi_i$ be the event that the image of $X_i$ under $\phi$ visits a new point at time instant $i$ and let $\tilde{\chi_i}$
be the event, that the image of $X_i$ under $\phi$ visits a new point at time instant $i$ and the increment $Y_{i+1}$ is in $\Delta$. Then the expectation
$$
E[\tilde{\chi_i}] = \mu(\Delta) E[\chi_i].
$$
Therefore
$$
\sum_{i=1}^n[\tilde{\chi_i}] = \mu(\Delta) \sum_{i=1}^n E[\chi_i].
$$
And by our  assumption, the latter is at least $pCn$. So (*) is satisfied.

The rest of our argument is similar to  the proof of Theorem 2.1 \cite{erschlerliouv}.
Given a random walk $(G,\mu)$, we fix two elements $\delta_1$ and $\delta_2$ in the support of $\mu$.
Let $(x_1,...,x_n)$ be the increments of a trajectory of our random walk. We assume that $1 \le i_1 < i_2 < \dots <i_k \le n$. We choose a set $w$ of $n-k$ elements of $G$:
$w=(x_1, x_2, \dots \hat{x}_{i_1}, \dots, \hat{x}_{i_j}, \dots x_n)$.
Here we use the convention that the hat over an element indicates its absence. 
Given $(w, i_1, \dots, i_k)$, we consider the set of trajectories of length $n$, denoted $T^{\delta_1,\delta_2}(w, i_1, \dots, i_k)$, where the increment at time $s$ is equal to $x_s$ unless $s$ is equal to some $i_j$, and that the increment at any time $i_j$ is either $\delta_1$ or $\delta_2$.
We also assume that $x_s$ are in the support of $\mu$
for all $s:1\le s \le n$.

We fix some constant $c,p>0$. 
Observe that the condition (*) and the cube independence assumption of the Lemma (cube independence) implies that the trajectory of $X_n$ admits a tuple of indices $A_n = (i_1, \dots, i_k)$ with the following properties 

a) the number of indices $k$ in $(i_1,...,i_k)$ satisfies $k\ge cn$ with probability $p$.

b) For any distinct $(\varepsilon_1, ..., \varepsilon_k) \in \{1, 2 \}^k$, the trajectories in $T^{\delta_1,\delta_2}(w, i_1, \dots, i_k)$ with increment $\delta_{\varepsilon_j}$ at step $i_j$ hit distinct endpoints at time instant $n$. (the number of such trajectories, and hence the number of endpoints is $2^k$).

Choose the indices $A_n = (i_1,...,i_k)$ by the following algorithm. For any $i<n$ if the event $\tilde{\chi_i}$ occurs, then the index $i$ is in $A_n$ otherwise it is not. From (*) it is immediate that property a) is satisfied. The assumption that $\Delta$ has the cube independent property along $\pi(S)$ immediately shows that property b) is satisfied.

We want to show that if all trajectories of length $n$ admit a tuple of indices $A_n$ satisfying a) and b), then the asymptotic entropy of the random walk $(G,\mu)$ is linear $n$. Moreover, we want to show the linear lower bound on the $\Delta$-restriction entropy at time $n$ for the two point set $\Delta = \{\delta_1, \delta_2\}$. Consider the probability measure $\nu$ on this set such that $\nu(\delta_1)= \mu(\delta_1)/\left(\mu(\delta_1) + \mu(\delta_2)\right)$ and $\nu(\delta_2)= \mu(\delta_2)/\left(\mu(\delta_1) + \mu(\delta_2)\right)$.

Our  assumptions show that the measure of the random walk $X_n$ after conditioning on its trajectory lying in the set $T^{\delta_1,\delta_2}(w, i_1, \dots, i_k)$ is a product measure $\nu^k$, and hence that 
the conditional entropy is  $k H(\nu(a,b))$. We also know that with positive probability $k \ge Cn$.
Using the fact the mean conditional entropy is not greater than the entropy, we obtain the first claim of the Lemma.

To prove 2) observe that a range of a transient random walk on $\mathbb{Z}^d$ is linear (see e.g Theorem 1.4.1 Spitzer [54]). In fact, the range of a transient random walk is linear for any random walk on a group (see Lemma 1 \cite{erschlerwreath}) and by a result of Varopoulos \cite{varopoulos} (see \cite{woessbook}, see also Appendix in \cite{erschlerfrisch} for further references)
any irreducible random walk on a group, which is not virtually $\mathbb{Z}$, $\mathbb{Z}^2$ or finite, is transient.
\end{proof}

\begin{corollary}
Let $d\ge 3$, let $\mu$ be a  finite entropy probability measure on $B_d(R)$, $R =\mathbb{Z}$ or $R =\mathbb{Z}/p\mathbb{Z}$, and let $S^+$ be a sub-semigroup of $B_d(R)$ generated by the support of $\mu$. Assume that 
$\pi(S^+)$ generates $\mathbb{Z}^d$ as a group.
Then the Poisson boundary of ($B_d(R),\mu)$ is non-trivial. In particular, any irreducible random walk on $B_d(R)$ or on $\mathbb{Z}^d \wr R$, has non-trivial Poisson boundary.
\end{corollary}

\begin{proof}

Now take a measure on the Baumslag group $G=B_d(R)$  ($R=\mathbb{Z}$ or $\mathbb{Z}/p\mathbb{Z})$ with the support generating $G$ as a subgroup. We know that the boundary of $(G,\mu)$
is canonically isomorphic to the boundary for a (non-atomic) affine combination $\nu$ of non-negative convolution powers $\mu$ (see \cite{kaimanovich83}, Theorem 4).
In particular, the boundary is non-trivial if and only if the latter boundary is non-trivial.
By Remark \ref{lem:semigroupBaum} we conclude that there exists $\nu$, which is an affine combination of convolution powers of $\mu$ such that $\nu$ satisfies  the assumption of Lemma 3.7. Combining the claim of that lemma with 
 the second claim of Lemma  \ref{lemme:corelemma} we see that $(G,\nu)$, and hence also $(G,\mu)$ has non-trivial Poisson boundary.
\end{proof}

\section{single polynomial case in three variables: $G_k(p)$ groups}

In this section, we will prove Theorem $B$.
As we have mentioned in the introduction, we say that  a Laurent polynomial $p$
in $d$ variables satisfies the spaced polynomial property if there exists some integer $N>0$ so that for any non-zero Laurent polynomial $u$ with coefficients in $0, \pm 1$, $p$ does not divide $u(x_1^N,\dots,x_d^N)$.

In the proof of Theorem B, we consider a  measure on our group $G_k(p)$, with the support of the measure generating the group. We remind the reader that the Poisson boundary does not change if we replace the measure by a non-trivial affine combination of convolution powers of the measure. In particular, we can assume that the resulting measure charges $e$.
Therefore, replacing the measure by such an affine combination of convolution powers, 
we can assume there exist $\delta_1$ and $\delta_2$ in the support of our measure,
with the same projection to $\mathbb{Z}^k$, and this projection has all coordinates divisible by $N$. 
Here $N$ is from the definition of spaced polynomial property (for the polynomial $p$).
We put $\Delta=\{\delta_1, \delta_2 \}$.
And our goal is to give a lower bound for $\Delta$-restriction entropy.

We start with a lemma.
\begin{lemma} \label{lem:monomialsdistinct} 
If $p$ has the spaced polynomial property, then any two distinct monomials $\prod_{i=1}^k x_i^{n_i}$ and $\prod_{i=1}^k x_i^{m_i}$ are not equal modulo $p$.  Here $n_1,...,n_k, m_1,...,m_k \in \mathbb{Z}$
\end{lemma}

\begin{proof} Seeking a contradiction assume the difference $\prod_{i=1}^k x_i^{n_i} - \prod_{i=1}^k x_i^{m_i}$ is divisible by $p$, then setting $r=\prod_{i = 1}^k x_i^{n_i - m_i}$ we see that $r-1$ is divisible by $p$. This means that $r^N-1$ is divisible by $q$ for any integer $N>0$ which contradicts the spaced polynomial property. 

\begin{definition}
We say that a Laurent polynomial $u$ is flat if the only nonzero coefficients are $\pm1$
\end{definition}

\end{proof}

Recall that  $G_k(p)$ is the group generated by the following matrices

\begin{equation}
\label{eq:defdelta}
   \delta=\left( \begin{array}{ccc}
1 & 1  \\
0 & 1
 \end{array} \right),
 \medspace \medspace
 M_{x_i}=
\left( \begin{array}{ccc}
1 & 0  \\
0 & X_i\\
 \end{array} \right).
\end{equation}
over the ring $R(X_1^{\pm 1},X_2^{\pm 1},...,X_k^{\pm 1})/I$,
where $p$ is a polynomial in $X_1$, $X_2$,...,$X_k$ and $I$ is the ideal
of $R(X_1^{\pm 1},X_2^{\pm 1},...,X_k^{\pm 1})$ generated by $p$.
In this section, we focus on the case $R=\mathbb{Z}$.

It is clear that Lemma 4.1 implies.

\begin{remark}
Assume that $p$ has the spaced polynomial property. Then the mapping 
that sends a matrix with lower right entry $\prod_{i=1}^k X_i^{\alpha_i}$ to $\left( \alpha_1, \alpha_2,...,\alpha_k \right)$
defines a homomorphism onto $\mathbb{Z}^k$. We denote this homomorphism by $\pi$.
\end{remark}

Now we prove

\begin{proposition}\label{prop:cubeindg3p}
Suppose that $p$ satisfies the spaced polynomial property with the constant
$N$.
Consider a homomorphism  $\pi:G_k(p) \to \mathbb{Z}^k$, and consider the preimage $H =\pi^{-1} ((N \mathbb{Z})^k) $ of the lattice $(N\mathbb{Z})^k$.
\begin{enumerate}
\item If $h_1,...,h_m \in H$, $\pi(h_i) \ne \pi(h_j)$ for $i\ne j$, and $\delta$ is as in \eqref{eq:defdelta} then the $m$ elements
$\gamma_i=h_i\delta h_i^{-1}$ are cube independent.

\item Moreover, if $\Delta =\{ \delta_1, \delta_2 \}$, $\delta_1 \ne \delta_2$ and $\pi(\delta_1) =\pi(\delta_2)$, then $\Delta$ is cube independent
along the image $\pi(H)$.

\item There exists a constant $C>0$, depending on the generating set of $G_k(p)$, such that for any integer $n>0$, the group $G$ admits $n^k$ commuting cube independent elements of length at most $Cn$.

\end{enumerate}
\end{proposition}
\begin{proof}

(1) For the first claim, note that the $\gamma_i$ are uni-upper-triangular matrices with monomials in the upper right corner, in particular, they commute. The first claim follows from observing that when $(\alpha_1,...,\alpha_k)$ and $(\beta_1,...,\beta_k)$ are distinct elements of $ \{0, 1\}^k$, the product $\prod_{i = 1}^m \gamma_i^{\alpha_i - \beta_i}$ is a uni-upper-triangular matrix with a nonzero flat Laurent polynomial in $x_1^N,...,x_k^N$ in the upper right corner. Because $p$ has the spaced polynomial property, this is a nonzero element of $G_k(p)$, so the $\gamma_i$ are cube independent. 

(2) To prove the second claim, we repeat the argument in the proof of the first claim above, replacing $\delta$ by $\delta_1^{-1} \delta_2$. Note that $\delta_1^{-1} \delta_2$ is a uni-upper-triangular matrix with a Laurent polynomial $q$ in the upper right corner, and that $q$ is not divisible by $p$ becasue $\delta_1^{-1}\delta_2 \neq 1$. Set $\rho_i =h_i \delta_1^{-1} \delta_2 h_i^{-1}$, and note that if $(\alpha_1,...,\alpha_k)$ and $(\beta_1,...,\beta_k)$ are distinct elements of $ \{0, 1\}^k$, the product $\prod_{i = 1}^m \rho_i^{\alpha_i - \beta_i}$ is a uni-upper-triangular with upper right entry equal to $q$ times a flat Laurent polynomial in $x_1^N,...,x_k^N$. Because $p$ has the spaced polynomial property it does not divide the flat Laurent polynomial, and we have already noted that $p$ does not divide $q$, so this product is not equal to the identity. Thus the $\rho_i$ are cube independent, and by Remark \ref{rem:cubeindependent} we see that $\Delta$ is cube independent along the image $\pi(H)$.

(3) Finally, since the balls in  $\mathbb{Z}^3$ have cardinality $\ge C n^3$, the third claim follows from the first one.

\end{proof}

Combining the second claim of Proposition \ref{prop:cubeindg3p} with the second claim of Lemma \ref{lemme:corelemma}, we obtain the statement of Theorem B.

Now consider the case of $p=1+x-y$. In the next section, we will show that $1+x+y$,
and hence also $p=1+x-y$ satisfy spaced polynomial property (for the constant $N=3$).
As we have already mentioned
the group $G_3(1+x-y)$ is isomorphic to the reduced torsion-free Baumslag group.
Therefore, as a particular case of Theorem $2$, we will be able to conclude that

\begin{corollary} \label{cor:restrictedBaumslag}
Any finite entropy, irreducible measure on 
the torsion-free restricted Baumslag group $\tilde{B}_2(\mathbb{Z})$ has nontrivial Poisson boundary.
\end{corollary}

\section{Spaced Polynomial Property.} \label{sec:algebraic}

\subsection{Preliminary observations about the spaced polynomial property.}

We will use the term "cyclotomic" for polynomials that can be multivariate:

\begin{definition} 
A polynomial $p(x_1,...,x_d)$ in the ring $\Z[x_1,...,x_d]$ is {\it cyclotomic} if it is irreducible and divides some polynomial of the form $x_1^{i_1}...x_d^{i_d}-1$.
\end{definition}

\begin{definition} \label{def:gencyc}
A Laurent polynomial in the ring in the ring $\Z[x_1^{\pm 1}, ..., x_d^{\pm 1}]$ is a {\it generalized cyclotomic} if it is irreducible and divides a difference of Laurent monomials $x_1^{i_1}...x_d^{i_d} - x_1^{j_1}...x_d^{j_d}$. 
\end{definition}
\noindent Note that because Laurent monomials are units in the ring of Laurent polynomials, generalized cyclotomics always divide $x_1^{k_1}...x_d^{k_d} - 1$, where $k_{\ell} = i_{\ell} - j_{\ell}$.

\begin{remark}
A Laurent polynomial is a generalized cyclotomic if and only if it is of the form $\phi(m(x_1,...,x_d)) o(x_1, x_2, .., x_d)$ where $\phi$ is a one variable cyclotomic polynomial, and $m$ and $o$ are monomials in $x_1^{\pm 1},..., x_d^{\pm 1}$. 
\end{remark}
\begin{proof}
    To see this note that $\mathrm{SL}_d(\mathbb{Z})$ acts on the Laurent monomials in $x_1,...,x_d$ by its action on their exponents, and that this action gives a natural action on the ring of Laurent polynomials. Any generalized cyclotomic divides some $x_1^{k_1}...x_d^{k_d} - 1$. We can choose an element $g$ of $\mathrm{SL}_d(\mathbb{Z})$ which sends $x_1^{k_1}...x_d^{k_d}$ to $x_1^m$ for some $m \in \mathbb{N}$.
    $p$ divides $x_1^{k_1}...x_d^{k_d} - 1$, so $g(p)$ divides $x_1^m - 1$ in the ring of Laurent monomials. Thus $g(p)$ is Laurent monomial times a one variable cyclotomic polynomial $\phi$. Thus $p = \phi(x_1^{k_1/r}...x_d^{k_d/r}) o(x_1,...,x_d)$ where $o$ is a Laurent monomial, and $r$ is the greatest common divisor of $k_1,...,k_d$.
\end{proof}

\begin{lemma}\label{lem:easycyclotomic}
If $p$ is a generalized cyclotomic, then it does not have the spaced polynomial property. 
\end{lemma}
\begin{proof}
If $p$ is a generalized cyclotomic, then it divides some $x_1^{k_1}...x_d^{k_d} - 1$. For any natural number $N$ the polynomial $p$ will also divide $(x_1^{k_1}...x_d^{k_d})^N -1$, so $p$ does not have the spaced polynomial property.

\end{proof}

As we have mentioned in the introduction, we believe that the converse statement to the claim of Lemma \ref{lem:easycyclotomic} is also true. 
It is not difficult to prove this in the case of (Laurent) polynomials in one variable:
\begin{lemma}
Let $\phi$ be a Laurent polynomial in $x$ with integer coefficients which is not a generalized cyclotomic.
Then $\phi$ satisfies the spaced polynomial property.
\end{lemma}
\begin{proof}
Consider the roots of $p$ over $\mathbb{C}$.
Recall that if an irreducible polynomial over the integers has all roots with absolute value one, then it is cyclotomic. The product of all nonzero roots of $p$ is an integer, so in fact, if $p$ is not generalized cyclotomic, then it has a root $\lambda \in \mathbb{C}$  with $|\lambda| > 1$. 

Choose $N$ such that $|\lambda ^N| >2$.
Let us prove that $p$ has the spaced polynomial property with respect to $N$. Seeking a contradiction, assume that $p$ divides a Laurent polynomial $u(x_1^N)$, where $u$ has coefficients in $\{-1, 0, 1\}$. Then in particular $\lambda$ is a root of $u(x_1^N)$, i.e. $u(\lambda^N) = 0$. If $u$ is degree $d$, then the highest degree term in $u(\lambda^N)$ has absolute value $|\lambda^{N}|^d$ and the sum of all other terms in $u(\lambda^N)$ have absolute value at most $\sum_{i=0}^{d-1} |\lambda^{N}|^i$. Because $|\lambda^N|>2$, we have $|\lambda^N|^d > \sum_{i=0}^{d-1} |\lambda^{N}|^i$, so the first term in $u(\lambda^N)$ has norm larger than all other terms combined, and $u(\lambda^N)$ cannot be $0$.  
\end{proof}

\subsection{Spaced polynomial properties of the polynomials for the Baumslag group.}

Let $\zeta=e^{\frac{2 \pi \i}{3}}$ be a primitive third root of unity.

\begin{lemma} \label{lem:computep}
$p(x,y)=\prod_{i,j=0}^2 (1+\zeta^i x + \zeta^j y)$ is equal to
$$
1+3x^3 + 3 y^3 + 3 x^6 +3 y^6 + 3 x^6 y^3 + 3 x^3 y^6+ x^9+y^9 - 21 x^3 y^3.
$$
\end{lemma}

\begin{proof} We compute

\begin{align} p(x,y)&= (1+x+y)(1+x +\zeta y) (1+x+\zeta^2 y) \nonumber\\
&\times (1+ \zeta x +y)(1+ \zeta x +\zeta y) (1+\zeta x+\zeta^2 y) \nonumber\\
&\times (1+\zeta^2 x+y)(1+\zeta^2 x +\zeta y) (1+\zeta^2 x+\zeta^2 y). \nonumber
\end{align}

Observe that for any $a,b$, we have
\[(a+b)(a+\zeta b)(a+\zeta^2 b)=a^3+b^3.
\]
In particular this implies
\begin{align}
(1+x+y)  (1+x+\zeta y) (1+x+\zeta^2 y)&=(1+x)^3+y^3, \nonumber \\
(1+ \zeta x +y)(1+ \zeta x +\zeta y) (1+\zeta x+\zeta^2 y)&=(1+\zeta x)^3+y^3, \nonumber \\
(1+\zeta^2 x+y)(1+\zeta^2 x +\zeta y) (1+\zeta^2 x+\zeta^2 y)&=(1+\zeta^2 x)^3+y^3. \nonumber
\end{align}

\vskip12pt
Expanding the right hand sides above and multiplying them together gives
\[p(x,y)=
(1+3x + 3x^2 +x^3 +y^3)   (1+3\zeta x + 3 \zeta^2 x^2 +x^3 +y^3 )  (1+3\zeta^2 x + 3 \zeta x^2 +x^3 +y^3)
\]

 Observe that non-zero coefficient monomials are of the form $x^a y^b$, where $b$,  and hence also $a$ are divisible by $3$.
 
 \begin{itemize}
 
\item For $1$, $x^9$ and $y^9$  the coefficient is $1$. 
 
\item For $y^3$ we have $3y^3$, 
 so that the coefficient for $x^3$ and $y^3$ is $3$.

 \item For $y^6$ we have $3y^6$, so that the coefficient for $x^6$ and $y^6$ are $3$.
 
 \item For $y^6 x^3$  we have $3 x^3 y^6$, so that the coefficients for $x^3 y^6$,  as well as for $x^6 y^3$ are $3$.
 
 \item Finally, for $x^3 y^3$ we have 
 $$
  3y^3 2 x^3 +  y^3 x^2 x ( 3 \zeta  3 \zeta + 3 \zeta^2 3 \zeta^2     +  9 \zeta^2 +9 \zeta + 9 \zeta + 9 \zeta^2)= (6+ 27 (\zeta+\zeta^2)) x^3 y^3= -21 x^3y^3
$$

\end{itemize}
\end{proof}

In the following proposition, we prove spaced polynomial property for the polynomial that we need for the restricted Baumslag group.

\begin{proposition} \label{prop:flatdivisor}
Let $u(x,y) \ne 0$ be Laurent polynomial with coefficients $\pm 1$ or $0$. Then $u(x^3,y^3)$ is not divisible by $1+x+y$. In other words, the polynomial $1+x+y$ satisfies the spaced polynomial property (for the constant $N=3$).
\end{proposition}

\begin{proof}[Proof of Proposition \ref{prop:flatdivisor}] 

Because monomials are units in the ring of Laurent polynomials we can assume without loss of generality that $u$ is a polynomial. Seeking a contradiction assume that $1+x+y$ divides $u(x^3,y^3)$. The polynomial $u(x^3,y^3)$ is invariant under the transformation $(x,y) \mapsto (\zeta_i x,\zeta_j y)$, so $1+\zeta_i x + \zeta_j y$  also divides $u(x^3, y^3)$ for  $i,j \in \{0,1,2\}$. Since  the (linear) polynomials $1+\zeta_i x + \zeta_j y$ are irreducible and not-proportional, we know that their product
$p(x,y)=\prod_{i,j=0}^2 (1+\zeta_i x + \zeta_j y)$ divides $u(x^3, y^3)$. In particular, there is a polynomial $q(x,y)$ with integer coefficients satisfying
\[p(x,y) q(x,y)=u(x^3,y^3).\]

Let $M$ be the maximal absolute value of any coefficient in $q(x,y)$, and consider all terms in $q(x,y)$ of the form $\pm M x^i y^j$. Let $i_0$ be the minimal power of $x$ occurring in any such term and choose some $j_0$ such that  $\pm M x^{i_0} y^{j_0}$ is a term in the polynomial $q(x,y)$.

The coefficients of $u(x^3,y^3)$ are the convolution of the coefficients of $p$ and $q$. Consider the coefficient of $x^{i_0+3}y^{j_0+3}$ in $u(x^3,y^3)$. This term is of the form $\pm21 M+S$, where $\pm 21 M$ corresponds to the monomial $21x^3 y^3$ in $p(x,y)$ multiplied by $\pm M x^{i_0} y^{j_0}$ in $q(x,y)$ and $S$ is the sum of all other terms in the convolution. 

Now we show $|S| \leq 21 M-7$. This comes from the fact (see Lemma \ref{lem:computep}) that the sum of absolute values of all remaining coefficients of $p$ is $21$ and  each of these coefficients is multiplied by a coefficient of $q$ which can be at most $M$. We subtract $7$ because coefficients of the monomials of $p$ whose exponent for $x$ is larger than $3$ (the norms of these coefficients sum to $7$) are multiplied by coefficients in $q$ with norm at most $M-1$.

This implies  that the coefficient of $x^{i_0+3}y^{j_0+3}$ in $u(x^3,y^3)$ has absolute value at least $7$ which gives a contradiction. 
\end{proof}

\begin{corollary}\label{cor:restrictedpol}  $1+x_1-x_2$ has the spaced polynomial property, considered as a polynomial in $x_1$ and $x_2$. It also has a spaced polynomial property considered as a polynomial in $x_1$, $x_2$, and $x_3$.

\end{corollary}

\begin{proof}
The first statement is simply a change of variables $x = x_1$, $y= -x_2$. For the second statement observe that after restricting our attention to the smallest power of $x_3$ to appear in some multiple of $1 + x_1 - x_2$, we are left with a multiple in the variables $x_1, x_2$ and can apply Proposition \ref{prop:flatdivisor}.
\end{proof}

The Corollary implies that the restricted Baumslag group $\bar{B}_2$ satisfies the assumption of Theorem B: that is, this group is isomorphic to $G_3(p)$, where
$p$ is a polynomial in three variables, satisfying the spaced polynomial property.

Now we give another lemma to address the general Baumslag group.

\begin{lemma} \label{lem:Baumslagalgebra} Let $f$ be a polynomial in $Y_1$ and $Y_2$
of the form $f=\sum_{a,b,c,d}^{\infty} C_{a,b,c,d} Y_1^{3a} (Y_1 +1)^{3b} Y_2^{c} (Y_2+1)^d$ where $C_{a,b,c,d} \in \{-1,0,1\}$, at least one $C_{a,b,c,d}$ is nonzero, and for any given $a,b, c$, there is at most one $d$ for which $C_{a,b,c,d} \neq 0$.
Then $f$ is not the zero polynomial.
\end{lemma}

\begin{proof}
Without loss of generality, we assume that there exists at least one triple $a,b,d$ such that $C_{a,b,0,d}$ is not equal to 0. 
(otherwise we divide $f$ by a power of $Y_2$).

Set $Y_2=0$ and then consider $f$
as a polynomial in $Y_1$.
We see that for any $a,b,c,d$ with $c>0$, the term $C_{a,b,c,d}Y_1^{3a} (Y_1 +1)^{3b} Y_2^{c} (Y_2+1)^d$ evaluates to $0$. We also see that for every $a,b$ there is at most one $d$ such that $C_{a,b,0,d}\neq 0$ 
and since $Y_2+1=0+1=1$, the associated product is $C_{a,b,0,d}Y_1^{3a}(Y_1+1)^{3b}Y_2^0(Y_2+1)^d=C_{a,b,0,d}Y_1^{3a}(Y_1+1)^{3b}$.
Thus after setting $Y_2=0$, we obtain a non-trivial flat polynomial in $Y_1^3,(Y_1+1)^3$ and therefore, by
Corollary \ref{cor:restrictedpol} with $x_1 = Y_1$ and $x_2 = Y_1 + 1$ this polynomial is nonzero. 

\end{proof}

The proposition below will ensure that

\begin{proposition}\label{prop:ciBaumslag}
Let $G=B_2(\mathbb{Z})$. Consider the map $\phi':G \to \mathbb{Z}^4$ that sends elements with $(y_1)^{\alpha_1} (y_1+1)^{\alpha_2}(y_2)^{\alpha_3}(y_2+1)^{\alpha_1}$ in the lower right corner of the matrix to $(\alpha_1, \dots, \alpha_4)$.
We denote by $\phi$ its projection to the first three coordinates $(\alpha_1, \alpha_2, \alpha_3)$.
Consider the lattice $L = 3\mathbb{Z} \times 3\mathbb{Z} \times \mathbb{Z} \subset \mathbb{Z}^3$ where $\alpha_1$ and $\alpha_2$ are both divisible by $3$ and put $H=\phi^{-1}(L)$. It is clear that $H$ is a finite index subgroup of $G$.
Consider some $\Delta=\{\delta_1, \delta_2\}$ such that
$\phi(\delta_1)= \phi(\delta_2)$ and $\delta_1 \neq \delta_2$.
We claim that $\Delta$ is cube independent along the image $\phi(H)$.
\end{proposition}

\begin{proof}
First, we prove this in the special case $\delta_1= \delta$, $\delta_2=e$.

To prove the cube independence along the image $\phi(H)$, first let $h_1,...,h_r$ be elements of $H$ with distinct images under $\phi$. We will set $\gamma_i = h_i \delta h_i^{-1}$, and show that $\gamma_1,...,\gamma_r$ are cube independent. First note that $\gamma_i$ is a uni-upper-triangular matrix whose upper right entry has the form $Y_1^{3k} (Y_1+1)^{3l} Y_2^{m} (Y_2+1)^{j_{k,l,m}}$ where $\phi'(h_i) = (3k, 3l, m, j_{k, l, m})$. In particular the $\gamma_i$ commute with one another. Let $(\alpha_1,...,\alpha_r)$ and $(\beta_1,...,\beta_r)$ be distinct elements of $\{0, 1\}^r$, it is enough to show that the product $\prod_{i=1}^r \gamma_i^{\alpha_i - \beta_i}$ is not the identity. 

The product $\prod_{i=1}^r \gamma_i^{\alpha_i - \beta_i}$ is a $2$ by $2$ matrix. The entry in the upper right corner of this matrix has the form 
$$
q = \sum_{k,l,m} \varepsilon_{3k,3l, m} Y_1^{3k} (Y_1+1)^{3l} Y_2^{m} (Y_2+1)^{j_{k,l,m}},
$$
where each tuple $(3k, 3l, m, j_{k, l, m})$ is of the form $\phi'(h_i)$ for one of the $h_i$. each $\varepsilon_{k,l,m}$ takes values in $\{-1, 0, 1\}$, in particular $\alpha_i - \beta_i = \varepsilon_{\phi(h_i)}$.
Without loss of generality all exponents $(k, l, m, j_{k, l, m})$ in the sum are positive, (if not multiply by a term $ (Y_1)^{3A} (Y_1+1)^{3B} Y_2^{C} (Y_2+1)^D$ to make them positive) so Lemma  \ref{lem:Baumslagalgebra} implies that $q \neq 0$ and we have cube independence.

Now consider the general case where $\delta_1$ and $\delta_2$ are not equal, but have the same image under $\phi$. This argument is quite similar. Let $h_1,...,h_r$ be elements of $H$ with distinct images under $\phi$. We will set $\rho_i = h_i \delta_1^{-1} \delta_2 h_i^{-1}$, and show that the $\rho_i$ are cube independent. First note that $\rho_i$ is uni-upper-triangular with upper right entry equal to $f(Y_1, Y_2) Y_1^{3k} (Y_1 + 1)^{3 l} Y_2^m (Y_2 + 1)^{j_{k, l, m}}$ where $\phi'(h_i) = (3k, 3 l, m, j_{k, l, m})$ and $f(Y_1, Y_2)$ is the upper right entry of $\delta_1^{-1} \delta_2$. 
Repeating the same argument from the special case above we get that $\prod_{i = 1}^r \rho_i^{\alpha_i - \beta_i}$ is a matrix with upper right entry
$$s = f(Y_1, Y_2) \sum_{k,l,m} \varepsilon_{3k,3l, m} Y_1^{3k} (Y_1+1)^{3l} Y_2^{m} (Y_2+1)^{j_{k,l,m}} = f(Y_1, Y_2) q,$$
where each $\varepsilon_{3k,3l, m}$ takes values in $\{-1, 0, 1\}$, and at least one of the $\varepsilon$ is nonzero. Note that $f(Y_1, Y_2)$ is not equal to zero, because $\delta_1^{-1} \delta_2 \neq e$, and without loss of generality we can assume the exponents $(3k, 3l, m, j_{k, l, m})$ are all positive. Lemma \ref{lem:Baumslagalgebra} shows that $q \neq 0$, and completes the proof. 

\end{proof}

Theorem A follows by combining Proposition \ref{prop:ciBaumslag} with the second claim of Lemma \ref{lemme:corelemma}.

\section{General linear groups. Subgroups with nontrivial Poisson boundary.}

The goal of this section is to prove Theorem C.
We start by recalling the notions of basic blocks
of a linear group and of the dimension of metabelian groups
from \cite{erschlerfrisch}.

\subsection{Basic block of upper-triangular matrices}\label{subsec:bb}

By $UT(n)$ we denote the group of uni-upper-triangular matrices i.e those upper triangular matrices with only $1$'s on the diagonal. 

Below we recall the notion of basic blocks. 
For pairs $(i,j)$, where $1 \le i < j\le n$ we consider the following partial order $U$: $(i,j) \leq_U (i',j')$ if $i \leq i'$ and $j \geq j'$. 

In the definition below,  given a $n\times n$ upper triangular matrix $G$ over some field $k$, we recall Definition 5.1 from \cite{erschlerfrisch}
of a basic block $B_{i,j}$ defined for $1 
\le i <j \le n$. 
This definition defines the matrices  over $k$ $B_{i,j}$, as explained below.

\begin{definition} \label{def:blocks}[Basic blocks]
For a fixed $(i, j)$, consider the set of matrices of the form
$$
G_{i,j}=
\left( \begin{array}{ccc}
g_{i,i} & 0  \\
0 &  g_{j,j}
 \end{array} \right), 
$$
where the matrices $g$ range over all elements of the group $G$.

If there is at least one  element in 
$G \cap UT(n)$ with non-zero entry at $(i,j)$ and with zero entries in all positions  $(i', j')<_U (i,j)$,  we say that $G$ admits a {\it valid basic block} 
$B_{i,j}$. Namely, 
 we consider the group generated
by all matrices of the form $G_{i,j}$
and by the matrix
$$
\delta=
\left( \begin{array}{ccc}
1 & 1  \\
0 &  1
 \end{array} \right),
$$
and we call this group the $(i,j)$  the basic block of $G$, and say that this basic block is valid. We denote this subgroup by  $B_{i,j}$.
Otherwise,  we say the $(i,j)$-block of $G$ is trivial (by definition,  $B_{i,j}$ is in this case  equal to a group consisting of the identity element).
\end{definition}

We also recall the definition of a slightly modified version of 
basic blocks $\tilde{B}_{i,j}$ (see Remark 5.2 \cite{erschlerfrisch}).
These are the groups generated by the matrices

$$
\left( \begin{array}{ccc}
1 & 0 \\
0 &  g_{i,i} g_{j,j}^{-1}
 \end{array} \right)
$$
and $\delta$.
(In Remark 5.2 \cite{erschlerfrisch} we used matrices above with a nontrivial diagonal element in the upper left corner. 
the group defined above is clearly an isomorphic one.)
We note that each modified basic block is the quotient of a basic block by a central subgroup.

Note that for generators $\tilde{G}_{i,j}$  of the modified basic block, its inverse is also of this form.
While \cite{erschlerfrisch} also defines basic blocks in the more general
context of nilpotent-by-Abelian groups, where the number of such blocks is not 
 necessarily finite, in the case of finitely generated linear groups the number of  blocks is finite.

\begin{remark} \label{rem:blockGI}
Consider a modified block $\tilde{B}_{i,j}$
of a linear group over field $\mathfrak{k}$ generated by $\delta$
and a finite number of $g_l^{\pm 1}$, where $g_l$ are matrices of the form $G_{i,j}$ ($1 \le  l \le k$).

There exists an ideal $I$ in
$\Z(X_1^{\pm 1},\dots,X_k^{\pm })$.
such that

$\tilde{B}_{i,j}$  is  equal to $G_k(I)$.
\end{remark}

Indeed, assume that
$$
g_l=
\left( \begin{array}{ccc}

1 & 0  \\

0 &  \alpha_l
 \end{array} \right).
$$
Consider the evaluation map from 
$\Z(X_1^{\pm 1},\dots,X_k^{\pm 1})$ to $\rm{k}$ sending
$X_l$ ($1 \le l \le k$) to $\alpha_l$.
Let $I$ be the kernel of this map. Since it is a kernel of a ring homomorphism to a field, it is a prime ideal.
We recall that by definition $G_k(I)$ is the group generated by the following matrices
\begin{equation}
   \delta=\left( \begin{array}{ccc}
1 & 1  \\
0 & 1
 \end{array} \right),
 \medspace \medspace
 M_{x_i}=
\left( \begin{array}{ccc}
1 & 0  \\
0 & X_i\\
 \end{array} \right).
\end{equation}
over the ring $\Z(X_1^{\pm 1},\dots,X_k^{\pm })/I$.
Here $1\le i \le k$.
Hence our  modified block is isomorphic
to $G_k(I)$.
\begin{remark}
By construction, the modified block is a quotient of $\mathbb{Z}^k \wr \mathbb{Z}$ if the field is of characteristic zero, and of $\mathbb{Z}^k\wr \mathbb{Z}/p\mathbb{Z}$ if the field is characteristic $p$.
\end{remark}

Now we recall the definitions of dimension for $k[A]$ modules and for metabelian groups (This is the definition we also used in \cite{erschlerfrisch}, see Definition 6.1).
Let $k$ be a field and $A$ be a finitely generated Abelian group.
Given a finitely generated $k[A]$ module $M$ and a subgroup $A'$ of $A$, consider $M$ as $k[A']$ module.
We say that $M$ is finitely $A'$ generated if $M$ is finite dimensional as a $k[A'] $ module. Let $d$ be the minimal number such that there exists $A'= \mathbb{Z}^d +C$, where $C$ is a finite Abelian group and  where $M$ is finitely generated as an $A'$ module. We say that $d$ is the $k$-dimension of $M$.

This is equivalent to the fact that there exists $A'= \mathbb{Z}^d$ with the property above, see Remark 6.3 in \cite{erschlerfrisch}. See also Remarks 6.13 and 6.14 in \cite{erschlerfrisch} where the relationship to Krull dimension is discussed.

We recall our convention for
dimension of metabelian group 
in Section 6.1 of \cite{erschlerfrisch}:

\begin{definition}
 Let $G$  be a metabelian group. Assume that $G$ is either torsion-free or $p$-torsion-by Abelian.
 Let $B$ be the commutator subgroup of $G$ and $A$ be  the abelianization of $G$. We have  a short exact sequence 
$$
1 \to B \to G \to A \to 1.
$$
Consider $B$ as a module over $\mathbb{Z}[A]$.
 If $B$ is a $p$-torsion group,  we  can also consider $B$ as a module over $\mathbb{Z}/p\mathbb{Z}  \mathbb[A]$. If $B$ is torsion-free,  we  can  consider $B \otimes_{\mathbb{Z}} \mathbb{Q}$  as a $\mathbb{Q} \mathbb[A]$ module.
 In case $B$ is a $p$-torsion group, we define the dimension of $G$ to be the dimension of $B$ considered as a module over $\mathbb{Z}/p\mathbb{Z} \mathbb[A]$.
 If $B$ is torsion-free, we define the dimension of $G$ as the dimension
 of $B \oplus_{\mathbb{Z}} \mathbb{Q}$ considered as a module over $\mathbb{Q}\mathbb[A]$.
\end{definition}

For any non-trivial quotient of the wreath product  
$\mathbb{Z}^k \wr \mathbb{Z}/p\mathbb{Z}$ the dimension is strictly smaller than $k$. If there is  exactly one non-trivial relation, and it is in the commutator group, then the dimension 
is exactly $k-1$ (see Example 6.20 in \cite{erschlerfrisch}). 
We also mention that by the first claim of  
Corollary 7.15 in \cite{erschlerfrisch}  we know that if
$G$ has is a metabelian $p$-torsion by finitely generated Abelian group
and the dimension is at most $2$, then for any centered finite second moment measure $\mu$ on G the
Poisson boundary of the random walk $(G, \mu)$ is trivial. 

In particular, we can conclude

\begin{remark} \label{rem:charpG3q}
Let $q$ be a non-zero polynomial.
Then $G_3(q, \mathbb{Z}/p\mathbb{Z})$ has dimension  equal to $2$.
A random walk on this group, defined by a symmetric finite second moment measure has trivial Poisson boundary.
\end{remark}

\subsection{Proof of Theorem C}

\begin{proposition} \label{prop:forTheoremC}
Let $1 \le m \le k-1$.
Let $I$ be an  ideal of  $\Z[x_1^{\pm 1} ,...,x_k^{\pm 1}]$, such that the minimal field $F$ containing the quotient ring $\Z[x_1^{\pm 1},...,x_k^{\pm 1}]/I$ has transcendence degree $m$.
Consider the subgroup of the multiplicative 
group of $\Z[x_1^{\pm 1},...,x_k^{\pm 1}]/I$ generated by
$x_i$:
$\langle x_1,...,x_k \rangle$.

Either $\langle x_1,...x_k \rangle$ is virtually $\Z^{m}$, or the group $G_k(I)$ has a subgroup which is of isomorphic to $G_{m+1}(p)$, where $p$ is an irreducible polynomial, $p$ is not a generalized cyclotomic,  and the minimal field $F$, containing $\Z[x_1^{\pm 1},...,x_{m+1}^{\pm 1}]/(p)$, has transcendence degree $m$. 
\end{proposition}

\begin{proof}

Because the smallest field $F$ containing $\Z[x_1^{\pm 1}, ... x_k^{\pm 1}]/I$ has transcendence degree $m$, there  exist $m$ elements among $x_i$
which are transcendentally independent over $\Q$. Changing the numeration, we can assume that these elements are $x_1$, $x_2$ \dots $x_m$. Thus the smallest field $F'$ containing $x_1$, $x_2$, \dots $x_{m}$ has transcendence degree $m$, and $F$ is a finite algebraic extension of $F'$. 

Either the multiplicative group $\langle x_1,..., x_k \rangle$ in $F$ is virtually $\Z^{m}$, in which case we are done, or there is some $x_j$ so that $\langle x_1, x_2, \dots, x_{m},  x_j \rangle$ has infinite index over $\langle x_1, x_2,...,x_m \rangle \cong \Z^{m}$. Change the numeration of the index set so that $j=m+1$. Consider the evaluation map $\phi_F: \Z[x_1^{\pm 1}, x_2^{\pm 1},  \dots, x_{m+1}^{\pm 1}] \to F$. Because the image of $\phi_F$ is contained in $F$ and thus has no zero divisors, the kernel $\ker(\phi_F)$ is a prime ideal of $\Z[x_1^{\pm 1}, x_2^{\pm 1}, \dots, x_{m+1}^{\pm 1}]$. We want to prove $\ker(\phi_F)$ is generated by a single polynomial in $\Z[x_1, x_2, \dots,  x_{m+1}]$.

Given an integral domain $R$ and a subset $M$ of non-zero elements of $R$ which is closed under multiplication, the localization of $R$ over $M$
is the subring of the field of fractions of $R$ whose denominator lies in $M$. We denote this localization  by $M^{-1}R$. 

First observe  that $\Q(x_1, x_2, \dots, x_{m})[x_{m+1}]$
is a principal ideal domain (since it is a polynomial ring over a field), and that this ring is equal to the localization of $R=\Z[x_1,x_2, \dots, x_{m+1}]$ at the set $M$ of non-zero elements of the subring $\Z[x_1, x_2, \dots, x_{m}]$.
Because the localization of a principal ideal domain is a principal ideal domain, $\Q(x_1, x_2, \dots, x_{m})[x_{m+1}^{\pm 1}]$ is also a principal ideal domain.

Because $\ker (\phi_F)$ is a prime ideal which does not intersect $\Z[x_1, x_2, \dots , x_{m}]$, its localization $M^{-1} \ker (\phi_F)  $  is the prime ideal $\Q(x_1, x_2, \dots, x_{m})[x_{m+1}^{\pm 1}] \ker(\phi_F)$ in $\Q(x_1, x_2, \dots, x_{m})[x_{m+1}^{\pm 1}]$ . Observe that this ideal is  generated by a single element $q \in \Q(x_1, x_2, \dots, x_{m})[x_{m+1}^{\pm 1}]$ since 
$\Q(x_1, x_2, \dots , x_{m})[x_{m+1}^{\pm 1}]$ is a principle ideal domain.

Recall that  Laurent polynomial rings are Noetherian unique factorization domains.
Thus since $\Z[x_1^{\pm 1},x_2^{\pm 1}, \dots, x_m^{\pm 1}]$ is a Noetherian unique factorization domain 
we know that the ideal  $\ker(\phi_F)$ has at least one irreducible prime element.
Denote this (irreducible) Laurent polynomial by $p$. Consider the ideal $(p) $ generated by $p$ in $\mathbb{Z}[x_1^{\pm 1}, x_2^{\pm 1}, \dots,  x_{m+1}^{\pm 1}]$. Note that the intersection of this ideal  with $\Z[x_1^{\pm 1}, x_2^{\pm 1}, x_3^{\pm 1}, \dots, x_m^{\pm 1}]$ is  0.
 Now observe that the localization of this ideal $M^{-1} (p)$ (with M as before) is a prime ideal contained in $M^{-1} \ker(\phi_F)$. 

 Since $\Q(x_1, x_2, \dots, x_{m})[x_{m+1}^{\pm 1}]$ is a principle ideal domain and since $M^{-1} (p)$ is prime,
 it must also be maximal, and hence $M^{-1} (p)$  is equal 
 equal $M^{-1} \ker(\phi_F)$. 
 
 Thus, since the localization induces a 1 to 1 correspondence between prime ideals in the ring $\Z[x_1^{\pm1},x_2^{\pm1}, \dots, x_m^{\pm1}, x_{m+1}^{\pm1 }]$ whose intersection with $\Z[x_1^{\pm 1}, x_2^{\pm 1},...,x_m^{\pm 1}]$ is empty and prime ideals in $\Q(x_1, x_2, \dots, x_{m})[x_{m+1}^{\pm 1}]$, we see that $\ker(\phi_F)$ must actually be equal to $(p)$.

Now we have to prove that $p$ is not a generalized cyclotomic. Seeking a contradiction assume that $p$ is generalized cyclotomic, Then $p$ divides some polynomial of the form $x_1^{i_1} x_2^{i_2} \dots x_{m + 1}^{i_{m + 1}} - x_1^{j_1} x_2^{j_2} \dots x_{m + 1}^{j_{m + 1}}$, so modulo $(p)$ we have $x_{m + 1}^{i_{m + 1} - j_{m+ 1}} = x_1^{j_1 - i_1} x_2^{j_2 - i_2} \dots x_{m}^{j_{m}-i_{m}}$. Here ${j_{m + 1}-i_{m + 1}} \neq 0$ since $\langle x_1, x_2, \dots , x_{m} \rangle$ is isomorphic to $\Z^{m}$. In other words some nonzero power of $x_{m+1}$ is equal to a Laurent monomial in $x_1$, $x_2$,  \dots $x_{m}$, and thus $\langle x_1, x_2, \dots, x_{m + 1} \rangle$ is a finite index extension of $\Z^{m}$ which 
is a contradiction with the assumption of the proposition. Thus $p$ is not a generalized cyclotomic.

Multiplying by positive powers of $x_1$, $x_2$, \dots  $x_{m+1}$ we can assume that the Laurent polynomial $p$ is actually an element in $\Z[x_1,x_2, \dots, x_{m+1}]$.

Finally consider the subgroup of $G_k(I)$ generated by $M_{x_1}$, $M_{x_2}$, \dots  $M_{x_{m+1}}$, and $\delta$. This subgroup is isomorphic to $G_{m+1}(p)$, for our polynomial $p$.

\end{proof}

\begin{corollary}
Let $1 \le m \le k-1$.
Let $I$ be an  ideal of  $\Z[x_1^{\pm 1} ,...,x_k^{\pm 1}]$, such that $\Z[x_1^{\pm 1},...,x_k^{\pm 1}]/I$ has dimension $m$.
Consider the subgroup of the multiplicative 
group of $\Z[x_1^{\pm 1},...,x_k^{\pm 1}]/I$ generated by
$x_i$:
$\langle x_1,...,x_k \rangle$.

Either $\langle x_1,...x_k \rangle$ is virtually $\Z^{m}$, or the group $G_k(I)$ has a subgroup which is of isomorphic to $G_{m+1}(p)$, where $p$ is an irreducible polynomial, $p$ is not a generalized cyclotomic,  and the minimal field $F$, containing $\Z[x_1^{\pm 1},...,x_{m+1}^{\pm 1}]/(p)$, has dimension $m$. 
\end{corollary}

\begin{proof}
Lemma $6.6$ from \cite{erschlerfrisch} implies that if $\Z[x_1^{\pm 1},...,x_k^{\pm 1}]/I$ has dimension $m$, then the minimal field $F$ 
 containing $\Z[x_1^{\pm 1},...,x_k^{\pm 1}]/I$ has transcendence degree $m$. Apply Proposition \ref{prop:forTheoremC}, and note that \cite{erschlerfrisch}[Lemma 6.11] implies that $\Z[x_1^{\pm 1},..., x_{m}^{\pm 1}]/(p)$ has dimension $m$. 
\end{proof}

Before proving Theorem C from the introduction, we recall its formulation.
Let $G$ be a finitely generated linear group
of characteristic $0$. Then Theorem C states that at least one of the following properties holds.
\begin{itemize}
    \item $G$ contains a free non-Abelian group as a subgroup. 
    \item $G$ has a finite index subgroup, of uppertriangular matrices, such that at least one valid block
    contains $\mathbb{Z}^3 \wr \mathbb{Z}$ as a subgroup.
    \item $G$ has a finite index subgroup, of uppertriangular matrices, such that at least one block
    contains $G_3(p)$ as a subgroup, for a polynomial $p$ with integer coefficients
    in $3$ variables which is irreducible over $\mathbb{Z}$ and not generalized cyclotomic.
   
    \item All finite second moment measures on $G$ have trivial Poisson boundary.
\end{itemize}

\begin{proof}[Proof of Theorem C]
Consider an amenable linear  group. By Tits alternative, if our linear group
does not have free non-abelian subgroups, then the group is virtually solvable. If we assume that our field is algebraically closed, then by Lie-Kolchin-Malcev theorem
(see e.g \cite{Robinsonbook}, Section 15.1)
we know that
$G$ contains a finite index subgroup, which is a subgroup of upper-triangular matrices. We consider this subgroup.
We need to show that it satisfies at least one of the claims
2), 3) or 4) of the Theorem.

The comparison criterion  (Theorem A) in \cite{erschlerfrisch} shows that the Poisson boundary for a finite entropy random walk on an upper-triangular group $G$ is trivial if and only if it is trivial for all associated random walks on its blocks.

We consider valid blocks of $G$. If there is at least one block of dimension $\ge 3$, 
we use Lemma 6.19  from \cite{erschlerfrisch}. That lemma in particular states that given a semi-direct product  of $A$ and $B$, with $B$ torsion-free and such that the dimension of the associated module is $d$, then this semi-direct product contains a $d$ dimensional wreath product $\mathbb{Z}^d \wr \mathbb{Z}$ as a subgroup. 
Basic blocks are semi-direct products of Abelian groups, so if their dimension $\ge d$, then they contain $\mathbb{Z}^d \wr \mathbb{Z}$ as a subgroup when $G$ is linear over a field of characteristic $0$,
and so in this case we have Claim 2) of the Theorem.

Otherwise, the dimension of all blocks is at most $2$. Suppose that  all blocks of dimension exactly $2$ are virtually $\mathbb{Z}^2 \wr \mathbb{Z}^r$. 

It is well known
that finite second moment random walks on $\mathbb{Z}^2 \wr \mathbb{Z}^r$ have trivial boundary. Indeed, the projection on $\mathbb{Z}^2$ is recurrent in this case, and (using triviality of the boundary of the exit measure random walk on the Abelian group sum B) thus the argument of Kaimanovich and Vershik \cite{kaimanovichvershik}
proves triviality of the boundary for any Abelian $B$.

If the dimension of a basic block is at most $1$, then by Corollary  7.15 of \cite{erschlerfrisch}, finite second moment random walks on this block are Liouville.
Therefore, if the dimension of all blocks is at most $2$ and those of dimension $2$ are virtually two dimensional wreath products 
then any finite second moment random walk on $G$ has trivial boundary.
Observe that the exit measure on a finite index subgroup of a measure with finite second moment also has finite second moment, so in this case we know that all finite second moment random walks on $G$ also have trivial Poisson boundary, and thus, in this case, we have claim 4) of the theorem.

If none of the situations described above happen, we know that a finite index subgroup of $G$ has at least one block of dimension $2$ that is not commensurable with $\mathbb{Z}^2 \wr \mathbb{Z}^r$. We want to apply proposition \ref{prop:forTheoremC} to this block. To see that the assumption of the proposition is satisfied note that
Lemma 6.22 \cite{erschlerfrisch} applied to a (modified) basic block implies the following:
if we consider the minimal field $F$, containing the generators of the block, and assume that the transcendence degree is $d$, then the dimension is d.

By Remark \ref{rem:blockGI} we thus know that our modified basic block is isomorphic to a group
$G_k(I)$. We apply Proposition \ref{prop:forTheoremC} for $m=2$ and conclude that
the modified block contains $G_3(p)$ as a subgroup,  for an irreducible polynomial $p$ which is not generalized cyclotomic.

We know that a modified block $\tilde{B}_{i,j}$ of $G$ contains a subgroup $\tilde{H}$ isomorphic to $G_3(p)$. Consider its generators $\tilde{g}_1, \tilde{g}_2, \tilde{g}_3$ and  $\delta$ (where $\delta$ is in the 
commutator subgroup of the upper triangular matrices). 
Recall that the modified block is a quotient of the corresponding block. 
 We lift $\tilde{g}_i$ to the block through the quotient map,
 we call these lifts $g_i$. We denote by $H$ the subgroup generated by $g_1$, $g_2$, $g_3$ and $\delta$.
Observe that the diagonal subgroup of $\tilde{H}$ is isomorphic to $\mathbb{Z}^3$.
Observe that the diagonal subgroup of $H$ is an Abelian group, generated by $3$ elements and admitting a quotient to $\mathbb{Z}^3$. Hence it is isomorphic to $\mathbb{Z}^3$. The kernel of the quotient map taking $H$ to $\tilde{H}$ consists only of diagonal elements, and because the diagonal subgroups of $H$ and $\tilde{H}$ are isomorphic, the quotient map is also an isomorphism. So $H \cong G_3(p)$.

\end{proof}

\end{document}